\documentclass[12pt]{article}
\usepackage[utf8]{inputenc}
\usepackage[dvips]{graphicx}
\usepackage{amssymb,color}
\usepackage{amsmath,amsthm}
\usepackage{epic}
\usepackage{tikz}
\usetikzlibrary{arrows.meta,positioning,calc,shapes.geometric,shadings,decorations.pathreplacing}
\usepackage{url}
\usepackage{graphicx}
\usepackage{soul}
\usepackage{cancel}
\usepackage[normalem]{ulem}
\usepackage{xcolor}

\usepackage{amsthm,amssymb,amsmath,amscd,graphicx, epsfig,hhline,enumerate}
\newtheorem{theorem}{Theorem}[section]
\newtheorem{lemma}[theorem]{Lemma}

\newtheorem{remark}[theorem]{Remark}
\newtheorem{defn}[theorem]{Definition}

\begin{document}

\title{Upper chromatic number of generalized quadrangles.}

\author{ 
Gabriela Araujo-Pardo$^1$\thanks{%
~garaujo@im.unam.mx}, 
Juli\'an Fres\'an$^2$\thanks{%
~julibeto@gmail.com},
\\
Linda Lesniak$^3$\thanks{%
~linda.lesniak@wmich.edu} 
and 
Mika Olsen$^2$\thanks{%
~olsen@cua.uam.mx\\ 
$^1${\small Instituto de Matem\'aticas, }
{\small Universidad Nacional Auton\'oma de M\'exico, }
{\small Campus Juriquilla, Quer\'etaro. }\\
$^2${\small Departamento de Matemáticas Aplicadas y Sistemas, }
{\small Universidad Aut\'onoma Metropolitana - Unidad Cuajimalpa, }
{\small Ciudad de M\'exico.}\\
$^3${\small Department of Mathematics, }
{\small Western Michigan University, }
{\small Kalamazoo, Michigan.}\\
{\small Research supported in part by PAPIIT-UNAM-M{\' e}xico IN113324, and  SECIHTI-M{\'e}xico under Project CBF2023-2024-552 and Project CBF-2025-G-1435.}}}
\maketitle

\begin{abstract}
We can regard a generalized quadrangle as a hypergraph in which points and lines are identified with vertices and hyperedges, respectively. A vertex coloring is said to be rainbow on a hyperedge if all vertices contained in that hyperedge are assigned pairwise distinct colors; if no hyperedge is rainbow, the coloring is termed rainbow-free. For a given hypergraph 
$H$, the maximum integer $k$ for which there exists a rainbow-free 
$k$-coloring is called the upper chromatic number of $H$, and is denoted by $\overline{\chi}(H)$.

In this work, we establish both lower and upper bounds for the upper chromatic number of certain generalized quadrangles. Moreover, we determine the exact values $\overline{\chi}(\mathcal{Q}(4,2))=6$ and $\overline{\chi}(\mathcal{Q}(4,3))=21$. 
\end{abstract}



\section{Introduction}
The notion of a rainbow-free coloring \cite{jnr} coincides with the notion of a proper strict coloring of a $\mathcal{C}$-hypergraph in the context of the recent theory of coloring mixed hypergraphs \cite{v}. Let $H$ be a hypergraph and $k$ be a positive integer. A \emph{$k$-coloring} of $H$ is a surjective mapping from $V(H)$ to a set of $k$ colors; that is $c:V(H)\to \{0,1,...,k-1\}$. The inverse images
of the colors are called the \emph{chromatic classes}; that is $C_i=\{v \in
V(H): c(v)=i\}$ for $i\in\{0,...,k-1\}$.  Given a  
$k$-coloring of $H$, an edge of $H$ is called \emph{rainbow},  if it is
totally multicolored (i.e., the coloring assigns pairwise distinct colors to its
vertices). A $k$-coloring of $H$ is said to be \emph{rainbow-free}, if it
contains no rainbow edges. The  \emph{upper chromatic number} of $H$, denoted
by $\overline{\chi}(H)$, is the largest integer $k$ for which there is a
rainbow-free $k$-coloring of $H$. We observe that, if  $\overline{\chi}(H)= k$ then $k+1$ is
the minimum integer with the property that every $(k+1)$-coloring
of $H$ contains a rainbow edge. This parameter is called the \emph{heterochromatic number}, and it was introduced by Arocha, Bracho, and Neumman-Lara in 1992 (see \cite{abn}). In this sense, the problem of determining the
upper chromatic number is considered 
an extremal anti-Ramsey problem.

In this work, we study hypergraphs associated with finite geometries. In this context, the sets of vertices and hyperedges of the hypergraph are the points and the lines of the finite geometry, respectively. We study the upper chromatic number for generalized quadrangles; the direct antecedent of this work is the study of the upper chromatic number of projective planes. There exists a lot of work related to these finite geometries; in particular, in 2003 Araujo-Pardo \cite{daisyPP} provided a lower bound for projective planes of order a power of two, using some sets called \emph{hyperovals} or \emph{$2$-sets}. These are sets of vertices that intersect any hyperedge in at least two vertices, and the {\emph{$2$-blocking set number}}, denoted by $\mathcal{X}_2$, is the minimum cardinality over all $2$-blocking sets of the hypergraph. Later, G. Bacs\'{o} and Z. Tuza, in 2007 \cite{bt}, found general bounds using also the notion of $2$-blocking sets.
It is not difficult to see that:
\begin{equation}\label{eq:natural} \overline{\chi}(H)\geq |V(H)|-\mathcal{X}_2(H)+1\end{equation} 
holds true for every hypergraph $H$, and it is an interesting problem to determine when this inequality is tight. G. Bacs\'{o}, T. H\'{e}ger, T. Sz\H{o}nyi, in 2013 \cite{bhs}, found exact values for projective planes of order $q^2$ for $q$ a prime power using the existence of Baer subplanes.

The upper chromatic number has been studied in many different contexts and has
been redefined several times under different names (see
\cite{abn,am,af,ess,fmo23,jks,k,ms,mnr} and references therein). More specifically, results
in projective planes appear, for example, in \cite{daisyPP,bhs,bt}.

It is important to say that we {consider} the generalized quadrangles, denoted by $\mathcal{Q}(4,q)$, and we use the geometrical terminology. We talk about points $\mathcal{P}$ and lines $\mathcal{L}$ of $\mathcal{Q}(4,q)$, not vertices and hyperedges. 
We say that a subset $\mathcal{X}\in \mathcal{P}$  is a {\emph{$2$-blocking set}} if $\mathcal{X}$ has at least two points in each line of $\mathcal{Q}(4,q)$, moreover we say that $\mathcal{X}$ {\emph{blocks the lines $\mathcal{L}$} }, and we said that a {\emph{line $\ell\in \mathcal{L}$ is {blocked} by $\mathcal{X}$}} if $\ell$ has at least two points in $\mathcal{X}$. A {\emph{$2$-blocking set}} in $\mathcal{Q}(4,q)$ is a set of points that block all the lines of  $\mathcal{Q}(4,q)$. And the {\emph{$2$-blocking number}} is the minimum cardinality taken over all the $2$-blocking sets. Notice that we can talk about minimal $2$-blocking sets that are $2$-blocking sets that do not contain $2$-blocking sets as subsets. There exist {authors} that call the $2$-blocking sets $2$-transversals (see for instance \cite{AMMV17}).

With this geometrical point of view, we say that the upper chromatic number of a generalized quadrangle is the maximum number of colors in a coloring of the points of $\mathcal{Q}(4,q)$ without rainbow lines. This parameter is denoted by $\overline\chi(\mathcal{Q}(4,q))$.

{If $\mathcal{X}$ is a $2$-blocking set} in $\mathcal{Q}(4,q)$, we can color {the points of $\mathcal{X}$} with one color and the rest of the points with different colors, and we have that: 
\begin{equation}\label{eq:natural2} \overline\chi({\mathcal{Q}(4,q)})\geq |\mathcal{P}-\mathcal{X}|+1.\end{equation}\label{boundgc}
This lower bound for the upper chromatic number can be maximazed by the $2$-blocking set number, but imagine that we have a $2$-blocking set of bigger order than the $2$-blocking set number, but it can be {colored} with more colors maintaining no rainbow lines or, even better, that we have a $2$-blocking set of minimum order that can be colored with more than one color. Then, we improve the bound given in (\ref{boundgc}).

As we said previously, the antecedent of this problem is the study of this parameter in projective planes: 

Let $q$ be a prime power. A projective plane of order $q$ is a finite geometry, denoted by $\mathcal{PG}(2,q)$, that satisfies the following conditions: any pair of points induces one and only one line, and any two lines intersect in at most one point. Moreover, in this finite geometry, there always exists a set of four non-collinear points. 
It is straightforward to prove that any point has the same number of incident lines, and that any line has the same number of points, and both values coincide. In particular, if $q$ is the order of the projective plane, this value is equal to $q+1$, and the projective plane has $q^2+q+1$ points and $q^2+q+1$ lines. 

It is an easy exercise to verify that the Fano plane $\mathcal{PG}(2,2)$ satisfies $\overline{\chi}(\mathcal{PG}(2,2))=3$ while the projective plane of order $3$,  $\mathcal{PG}(2,3)$ satisfies 
$\overline{\chi}(\mathcal{PG}(2,3))=6$. For $q=p^h$, with  $p$ a large enough prime, it
has been proved that concerning $\mathcal{PG}(2,q)$ the equality holds in (\ref{eq:natural}), and it is attained only by trivial colorings  \cite{bhs}.

In this work, we take a step further in finite geometries. We investigate the upper chromatic number of generalized quadrangles. Notice that, as in the projective planes, the generalized quadrangles exist when $q$ is a prime power, and any point has $q+1$ incident lines and any line has $q+1$ points. Moreover, in this case, the number of points is equal to $q^3+q^2+q+1$ and also the number of lines. Both projective planes and generalized quadrangles belong to the class of finite geometries; they are part of the family of generalized polygons, as we will explain in Section \ref{sec:preliminaries}.

This paper is organized as follows: In Section \ref{sec:preliminaries}, we give definitions and results that will be used throughout the article, and in Section \ref{results} we state the main results that we will prove later in the rest of the sections. In Section \ref{lowerbounds} we prove the lower bounds of the upper chromatic number depending on the value of $q$. In Section \ref{upperbound}, we prove the general upper bound for any generalized quadrangle. In Section \ref{q2y3} we prove the exact value of the upper chromatic number for the generalized quadrangles of order $2$ and $3$. For $q=2$ we use a description known as the Sylvester model (see \cite{sylvester}). Then, we prove that $\overline{\chi}(\mathcal{Q}(4,2)) = 6$ and $\overline{\chi}(\mathcal{Q}(4,3)) = 21$. 
Finally, in Section \ref{sec.conclandfuture} we give conclusions and propose future work.

\section{Preliminaries}
\label{sec:preliminaries}

As we said in the introduction, projective planes and generalized quadrangles belong to the class of finite geometries; moreover, they are part of the family of generalized polygons. We will delve a bit deeper into this concept:
 
\begin{defn}[\cite{KS20}]\label{kiss}
Let ${\mathcal{G}}=({\mathcal{P}},{\mathcal{L}},I)$ be finite point-line incidence geometry and $n\geq 2$ be an integer.

A chain of length $h$ is a sequence of $h+1$ elements $x_0,x_1\ldots,x_h\in {\mathcal{P}}\cup {\mathcal{L}}$ such that
$$x_0 I x_1 I \ldots Ix_h$$
We said that the chain joins $x_0$ and $x_h$ and that ${\mathcal{S}}$ is connected if for any two elements of ${\mathcal{P}}\cup {\mathcal{L}}$ there exits a chain joining them. The distance of the two elements $x$ and $y$ of ${\mathcal{P}}\cup {\mathcal{L}}$, denoted by $d(x,y)$, is the length of the shortest chain joining them. 

Then, ${\mathcal{G}}$ is called a {\emph{generalized $n$-gon}} if it is connected and satisfies the following axioms: 

\begin{itemize}
\item $d(x,y)\leq n$ for all $x,y\in {\mathcal{P}}\cup {\mathcal{L}}$.
\item If $d(x,y)=k<n$, then there exists a unique chain of length $k$ joining $x$ and $y$. 
\item For all $x\in {\mathcal{P}}\cup {\mathcal{L}}$ there exists $y\in {\mathcal{P}}\cup {\mathcal{L}}$ such that $d(x,y)=n$. 
\end{itemize}

\end{defn}

\begin{defn}[\cite{KS20}]\label{incidencegraphdef}
The incidence graph (sometimes called {\emph{Levi Graph}}) of a point-line incidence geometry is a bipartite graph whose two classes of vertices correspond to the set of points and the set of lines of the geometry, and two vertices are adjacent if and only if the corresponding point-line is incident in the geometry. 
\end{defn}

The following results are given in \cite{VM98}: 

\begin{lemma}[\cite{VM98}]\label{VMLemma}
A geometry ${\mathcal G} = (\cal{P},\cal{L},{\bf I})$ is a generalized $n$-gon if and
only if the incidence graph of ${\mathcal G}$ is a connected graph of diameter $n$ and
girth $2n$ in which each vertex is incident with at least three edges.
\end{lemma}

With Definition \ref{kiss}, a projective plane is a generalized $3$-gon and a generalized quadrangle is a generalized $4$-gon. Furthermore, the incidence graph of a generalized quadrangle $\mathcal{Q}(4,q)$ is a bipartite graph $(\mathcal{P},\mathcal{L})$, where $\mathcal{P}$ is the set of points and $\mathcal{L}$ is the set of lines. 

The incidence graph of a generalized quadrangle $\mathcal{Q}(4,q)$ is a $(q+1,8)$-Moore graph and it exists when $q$ is a prime power, has order  $2(q+1)(q^2+1)$, is $(q+1)$-regular, has diameter $4$ and girth $8$. 

Another structure that we will use in this paper is {\emph{ovoids}}, denoted by $\mathcal{O}$, which are maximal subsets of non-collinear points in $\mathcal{Q}(4,q)$. It is well known that the cardinality of $\mathcal{O}$ is equal to $q^2+1$. Moreover, in the incidence graph of $\mathcal{Q}(4,q)$, ovoids induce a maximal set of vertices mutually at distance four. 
For more information on Moore graphs or generalized quadrangles, see \cite{EJ13,KS20,LibroRojo,VM98}.

Given a 2-blocking set $B$ of $\mathcal{Q}(4,q)$, we say that a line is \emph{covered} by $B$ if it has exactly two neighbors in $B$, and if it has more than two neighbors in $B$  we say that it is \emph{extra-covered} by $B$. 

The following gives us an easy but useful proof that the minimum number of a $2$-blocking set is $2(q^2+1)$.

\begin{lemma} \label{blockingset} Let $B$ a blocking set of a generalized guadrangle $\mathcal{Q}(4,q)$, then $|B|\geq 2(q^2+1)$.
\end{lemma}

\begin{proof} Let $G$ the incidence graph of $\mathcal{Q}(4,q)$  with two partite sets of vertices: the points $\mathcal{P}$ and the lines $\mathcal{L}$. As each line $\ell\in\mathcal{L}$ is covered or extra-covered by $B$, the sum of the degrees of the set $\mathcal{L}$ is at least $2(q^3+q^2+q+1)$, and since each vertex has degree $q+1$, the sum of the degrees of the set $B$ is $|B|(q+1)$. Thus $2(q^3+q^2+q+1)\le |E(G)|=|B|(q+1)$ and $|B|\ge 2(q^2+1)$.
\end{proof}

In \cite{2ovoid} the author cited two papers \cite{Cossi,D2000} that prove the existence of $2$-blocking sets, calling {\emph {$2$-ovoids}}, with the property that it intersects each line in exactly two points. The $2$-ovoids has size $2(q^2+1)$ and only exists if $q$ is even or $q=3$. 
Then, this family attains the lower bound, and the $2$-blocking set number is exactly $2(q^2+1)$.

For $q\not=3$ an odd prime power, we have a blocking set of size $2q^2+q$. In the following, we explain in detail the configuration given in \cite{fmo23} to obtain this $2$-blocking set. It is interesting to note that this configuration is a kind of generalization of a configuration obtained in \cite{daisyPP} for projective planes, which was used in that article to provide a lower bound on the upper chromatic number of projective planes. The authors of these two papers call these configurations the \emph{Daisy Structures}.

In \cite{fmo23} the authors defined a Daisy structure that partitions the points of a generalized quadrangle of order $q$. The petals are $q$ ovoids $O_0, O_{1}, O_{2}, \ldots, O_{q-1}$ that intersects in exacly one point, the center,  and the $q+1$ stems are lines $T^*_\rho, T^*_0, T^*_1, \ldots, T^*_{q-1}$ that intersects in the center, see Figure \ref{fig:daisy}. 

The following result uses this Daisy structure. Let $c$ be the center of the Daisy structure, let $O'_k=O_k-c$ for $1\le k\le q-1$ and let $T_k=T^*_k-c$ for $k\in\{\rho, 0, \ldots ,{q-1}\}$.  

\begin{theorem}\cite{fmo23}\label{Teo:Daisy2}
Let $G$ be a $(q+1,8)$-Moore graph (obtained from a classic quadrangle) with bipartition $(\mathcal{P}, \mathcal{L})$. The sets $O_k$ are ovoids of $G$. The set $\mathcal{D}= \{ O_0, O^{'}_{1}, O^{'}_{2}, \ldots, O^{'}_{q-1}, T_\rho, T_0, T_1, \ldots , T_{q-1}   \}$ is a partition of $\mathcal{P}$.
\end{theorem}

Moreover, the description given in \cite{fmo23} proves that any line of the generalized quadrangle is a stem or intersects all the petals in one and only one point and has one point in one stem. With this description, we can see that if $S$ is the set of points of $\bigcup_{i=0}^{q-1} T_i \cup T_\rho$, then the set of points $B=\mathcal{O}^{'}_{1} \cup{S}$ is a $2$-blocking set of size $q^2+q(q+1)=2q^2+q$ (see the shaded part in Figure \ref{fig:daisy}). 

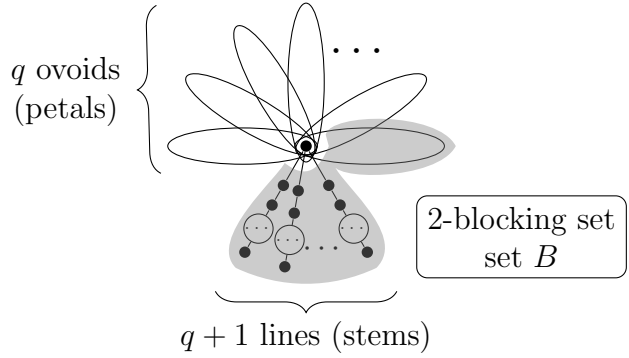
\begin{figure}
    \centering
\begin{tikzpicture}[scale=.6,
    dot/.style={circle, fill=black, inner sep=0pt, minimum size=0.15cm},
    cnode/.style={circle, draw, fill=white, inner sep=0pt, minimum size=0.20cm}
]
    \draw[decorate,
decoration={brace,amplitude=8pt}]
(-3.2,-.6)--(-3.3,3.1);

\node[align=center]
at (-5.3,1.2)
{$q$ ovoids\\ (petals)};

\draw[decorate,
decoration={brace,mirror,amplitude=8pt}]
(-2,-3.2)--(2,-3.2);

\node at (0,-4.2)
{$q+1$ lines (stems)};
    \node[
draw,
rounded corners,
fill=white,
align=center,
inner sep=4pt,
] at (4.75,-2)
{2-blocking set\\ set $B$};

    \draw (120:1.4) ellipse [x radius=0.4, y radius=1.65, rotate=30];
    \node at (60:2.4) {\Large$\cdots$};
    \draw (150:1.4) ellipse [x radius=0.4, y radius=1.65, rotate=60];
    \draw (30:1.4) ellipse [x radius=0.4, y radius=1.65, rotate=-60];
    \draw (90:1.4) ellipse [x radius=0.4, y radius=1.75];
    \draw (0:1.4) ellipse [x radius=0.4, y radius=1.65, rotate=-90];
    \draw (180:1.4) ellipse [x radius=0.4, y radius=1.65, rotate=90];
    
    \node[dot] (A) at (0,0) {};
    
    \node[dot] (B1) at (240:1) {};
    \node[dot] (B2) at (240:1.5) {};
    \node[cnode] (B3) at (240:2.1) {\tiny$\cdots$};
    \node[dot] (B4) at (240:2.7) {};
    
    \node[dot] (C1) at (260:1) {};
    \node[dot] (C2) at (260:1.5) {};
    \node[cnode] (C3) at (260:2.1) {\tiny$\cdots$};
    \node[dot] (C4) at (260:2.7) {};
    
    \node[dot] (D1) at (300:1) {};
    \node[dot] (D2) at (300:1.5) {};
    \node[cnode] (D3) at (300:2.1) {\tiny$\cdots$};
    \node[dot] (D4) at (300:2.7) {};

    \draw (A) -- (B1) -- (B2) -- (B3) -- (B4);
    \draw (A) -- (C1) -- (C2) -- (C3) -- (C4);
    \draw (A) -- (D1) -- (D2) -- (D3) -- (D4);

    \node at (280:2.3) {$\cdots$};

    \begin{scope}
\fill[gray,opacity=.4]
  (30:.5)
  .. controls (55:.9) and (15:3) .. (0:3.3)
  .. controls (345:3) and (300:.9) .. (330:.7)
  .. controls (310:.9) and (320:3) .. (300:3.1) 
  .. controls (280:3.2) and (260:3.2) .. (240:3.1)
  .. controls (223:3) and (228:.9) .. (218:.6)
  .. controls (270:.75) and (300:.6) .. (0:.5)
      -- cycle;
\end{scope}
\end{tikzpicture}
    \caption{Daisy structure of $\mathcal{Q}(4,q)$}
    \label{fig:daisy}
\end{figure}

In Section \ref{lowerbounds} we give lower bounds for the upper chromatic number of $\mathcal{Q}(4,q)$ that depend on the size of the $ 2$-blocking sets.


\section{Main Results}\label{results}

In this section, we state our main results: the lower bounds depending on the existence of the $2$-blocking sets that we will use to find rainbow-free colorings for the generalized quadrangles of order $q$, and the upper bound is the same for all the cases. We also give two exact values for $q=\{2,3\}$ that we will prove with different techniques in Section \ref{q2y3}. 

\begin{theorem}\label{mainthm}
Let $q$ be a prime power. Then we have that: 
\begin{enumerate}

\item If $q\geq 5$ and odd prime power, we have that:
$$q^3-q^2+2\le \overline{\chi}(\mathcal{Q}(4,q))\le q^3-q^2+2q-3. $$

\item If $q$ is an even prime power or $q=3$, we have that: 
$$q^3-q^2+q\le\overline{\chi}(\mathcal{Q}(4,q))\le  q^3-q^2+2q-3.$$

\item If $q=2$, then $\chi(\mathcal{Q}(4,2))=6$.

\item If $q=3$, then $\chi(\mathcal{Q}(4,3))=21$.
\end{enumerate}
\end{theorem}

\section{Lower bounds for the upper chromatic number of generalized quadrangles of order $q$}\label{lowerbounds}

In this section, we give two different lower bounds for the upper chromatic number of generalized quadrangles of order $q$, depending on the order of the $2$-blocking sets. 
In fact, we have the following remark using the results enumerated in the introduction and given in \cite{daisyPP,2ovoid,Cossi,D2000}.

\begin{remark}\label{2blocking}
 Depending on the value of $q$ we have two different bounds for the $2$-blocking set number:
\begin{enumerate}
\item For $q\geq 5$ an odd prime power, we do not know if a $2$-ovoid exists. But, we have a $2$-blocking set given in \cite{fmo23} of cardinality $2q^2+q$. Consequently, an upper bound for the $2$-blocking set number is $2q^2+q$.
\item If $q$ is an even prime power or $q=3$, the generalized quadrangles have {\it{$2$-ovoids}} of cardinality $2(q^2+1)$. The {\it{$2$-ovoids}} are $2$-blocking sets that intersect all the lines in exactly two points (see \cite{2ovoid}). Then, $2(q^2+1)$ is an upper bound for the $2$-blocking set number.
\end{enumerate}
\end{remark}

Appliying the Remark \ref{2blocking} to the equation (\ref{eq:natural}) we have the following theorem: 

\begin{theorem}\label{lowerboundsthm} Let $\mathcal{Q}(4,q)$ be a generalized quadrangle of order $q$. Then we have that:
\begin{enumerate}
    \item If $q\geq 5$ and odd prime power, then:
$$q^3-q^2+2\le \overline{\chi}(\mathcal{Q}(4,q))$$

\item If $q$ is an even prime power or $q=3$, we have that: 
$$q^3-q^2+q\le\overline{\chi}(\mathcal{Q}(4,q))$$
\end{enumerate}
\end{theorem}

The proof is straightforward, using the two values given in the Remark \ref{2blocking} and the fact that the number of points of $\mathcal{Q}(4,q)$ is $q^3+q^2+q+1$.

\section{Upper bound for the upper chromatic number of generalized quadrangles of order $q\ge4$}
\label{upperbound}

In this section we give a general upper bound for $\mathcal{Q}(4,q))$ for $q\geq 3$

\begin{theorem}\label{upperboundthm} Let $q\geq 4$ be a prime power. Then 
$$\overline{\chi}(\mathcal{Q}(4,q))\le q^3-q^2+2q-3.$$
\end{theorem}

To prove Theorem \ref{upperboundthm}, we will use the incidence graph of $Q(4,q)$, given in Definition \ref{incidencegraphdef} and Lemma \ref{VMLemma}.  It is a $(q+1,8)$-cage that is a bipartite graph with two partite sets of vertices, the points $\mathcal{P}$ and the lines $\mathcal{L}$ of $\mathcal{Q}(4,q)$. For this reason, throughout this proof, we will say that we have points and lines as the vertices of the incidence graph. Recall also that the incidence graph has diameter 4 and girth 8, and that is the minimum $(q+1)$-graph with these properties. 

Suppose that there is a rainbow-free $c$-coloring of $\mathcal{Q}(4,q)$ that attains the upper chromatic number, that is $\overline{\chi}(\mathcal{Q}(4,q))=c$. 
Notice that we only color the vertices of the bipartite graph that correspond to points, and that the union of all chromatic classes with more than one element induces a $2$-blocking set $B\subset \mathcal{P}$ of ${\mathcal{Q}(4,q)}$.

Let $G = (B, \mathcal{L})$ be the subgraph of the incidence graph of ${\mathcal{Q}(4,q)}$, induced by $B$. Observe that, in this subgraph, if $C$ is a chromatic class of the coloring, then the subgraph $(C, N_G(C))$ is connected; otherwise, we could obtain a rainbow-free coloring of $\mathcal{Q}(4,q)$ with more than c colors. Thus, $C\subset B_i$, where  $G_i=(B_i, \mathcal{L})$ is a connected component of $G$. Let $G_i=(B_i,\mathcal{L}_i)$, for $i\in \{1,\ldots,s\}$ be the connected components of $G=(B,\mathcal{L})$. Since $B$ is a $2$-blocking set, each line $\ell\in\mathcal{L}$ contains at least two points in one of the sets $B_i$, and the rest of their neighbors are also in $B_i$ or in $\mathcal{P}-B$.

To continue, we will analyze the structure of these subgraphs, and we will prove Lemma \ref{Lem:CotaB}, Lemma \ref{Lem:CotaB_q-1}, and Lemma \ref{lem:B-E*}, to finalize with the proof of Theorem \ref{upperboundthm}.

Notice that if a line is either covered or extra-covered by some set $B_i$, it is also covered or extra-covered by $B$.  

If we have extra-covered lines, then we delete edges until each line is covered by some  $B_i$. 
We denote such a set of deleted edges by $E^*$. Note that $E^*$ is not unique and given a set $E^*$, we define $G^*=(B,\mathcal{L})-E^*$. 

First, we determine some bounds on the cardinality of a $2$-blocking set of the lines and its relation with the structure of $G^*$. 

\begin{lemma}\label{Lem:CotaB}
    Let $B$ be a 2-blocking set. If $G=(B,\mathcal{L})$ and $|B|=2(q^2+1)$ then:
    \begin{enumerate}
    \item $d_G(p)=q+1$ for every vertex $p\in B$ and $d_G(\ell)=2$ for every vertex $\ell \in \mathcal{L}$,
        \item $E^*=\emptyset$, and 
    \item For every connected component $G_s=(B_s,\mathcal{L}_s)$, $B_s$ has at least $2q+2$ points.
    \end{enumerate}
\end{lemma}   
    
\begin{proof}
    Suppose that $|B|=2(q^2+1)$. Since $2(q^3+q^2+q+1)=(2q^2+2)(q+1)$ and by definition, $d_G(p)=q+1$ for every vertex $p\in B$, we conclude that $d_G(\ell)=2$ (any line $\ell$ is covered by $B$). And, we also conclude that $E^*=\emptyset$. 
        
    Now, suppose that $G$ is not connected, and let $G_s=(B_s,\mathcal{L}_s)$. 
    For any point $p\in B_s$, since $N(p)\subset \mathcal{L}_s$ and $d_G(p)=q+1$, it follows that $N(p)=\{\ell_0,\ell_1,\dots,\ell_q\}\subset \mathcal{L}_s$. Each line $\ell_i\in \mathcal{L}_s$ satisfies that $d_{{G}_s}(\ell_i)=2$; then, let $p_i\in B_s\cap N(\ell_i)\setminus p$, for each line $\ell_i$. Clearly, as $G$ has no 4-cycles, $p_i\neq p_j$ for $i\neq j$. 
    
    Moreover, for an arbitrary but fixed integer $1\le\alpha\le q$, take one of these points $p_{\alpha}$ and take all its neighbors in $\mathcal{L}_s$ different of $\ell_{\alpha}$, that is $N(p_{\alpha})\setminus{\ell_{\alpha}}=\{\ell^{\alpha}_1,\ell^{\alpha}_2,\dots,\ell^{\alpha}_q\}$. Each line $\ell^{\alpha}_j$ must have at least another point in $B_s$. 
    In consequence, we obtain $q$ points more $\{p^{\alpha}_1,p^{\alpha}_2,\dots,p^{\alpha}_q\}\in B_s$ such that $p^{\alpha}_j\in N(\ell^{\alpha}_j)\setminus p_{\alpha}$. Since $G$ has no $6$-cycles, these points also must be different for each $\ell^{\alpha}_j$ and also different from $\{p,p_0,p_1,\dots,p_q\}$.
     
     Thus, $\{p,p_0,p_1,\dots,p_q,p^{\alpha}_1,p^{\alpha}_2,\dots,p^{\alpha}_q\}\subset B_i$ and  $|B_i|\ge 1+q+1+q=2(q+1)$.
\end{proof}

We use the following definition throughout the proof of Theorem \ref{upperboundthm}. 

\begin{defn}\label{def:projection}
Let $G$ be a bipartite graph with partite sets $U$ and $V$. The \emph{$U$-partite projection of $G$}, denoted as $G_U$, is a graph whose vertex set is $U$ and two distinct vertices $u_1, u_2 \in U$ are adjacent in $G_U$ if and only if $N(u_1)\cap N(u_2)\neq\emptyset$.     
\end{defn}

As an observation, it is easy to see that there exist different bipartite graphs with the same $U$-projection. Using the Definition \ref{def:projection} and the results given in Lemma \ref{Lem:CotaB}, we have the following remark.

\begin{remark}\label{Rem:B}
    Let $q\ge 3$ be a prime power, and let $G=(B,\mathcal{L})$ be the bipartite graph with $N(B)=\mathcal{L}$ such that every vertex $\ell\in \mathcal{L}$ covered by $B$. Then, if $H$ is the $B$-projection of $G$, each edge of $H$ corresponds to a line in $\mathcal{L}$ and $G$ can be obtained by subdividing each edge of $H$. 
    Notice that, Lemma \ref{Lem:CotaB} implies that this happens when $|B|=2(q^2+1)$. 
 
\end{remark}
{In what follows, we repeatedly use the $B$-projection of $G$ and Remark~\ref{Rem:B}.}

\begin{lemma}\label{Lem:CotaB_q-1}
  If $B$ is a 2-blocking set of order $2(q^2+1)$ such that $G=(B,\mathcal{L})$ has $q-1$ connected components, then $|B_i|\ge2q+4$ for each component $G_i=(B_i,\mathcal{L}_i)$.
\end{lemma}

\begin{proof}
     By Lemma \ref{Lem:CotaB}, $B_i$ has at least $2q+2$ points and $E^*=\emptyset$.
     
     If $G$ has $q-1$ components, for a contradiction, assume that there exists a component $G_i$ with $|B_i| = 2q + 2$. 
    Since every point in $B_i$ has degree $q+1$ in the incidence graph, and every line in $L_i$ is covered by $B_i$ (it has degree 2 in $G_i$), counting the edges between $B_i$ and $\mathcal{L}_i$, we obtain $2|L_i| = |B_i|(q+1) = (2q+2)(q+1) = 2(q+1)^2$. Thus, $|L_i| = (q+1)^2 = q^2 + 2q + 1$.
    
Let $R=\mathcal{P}-B_i$ and consider the bipartite graph $H = (R, \mathcal{L}_i)$ induced by the set of points $R$ and the lines in $\mathcal{L}_i$. 
Each line $\ell \in \mathcal{L}_i$ is covering by $B_i$ (that is $N_{G_i}(\ell)=2$). Since every line in $\mathcal{Q}(4,q)$ has $q+1$ points, the remaining $q-1$ points must be in $R$. Therefore, $d_H(\ell) = q - 1$ for every line $\ell \in \mathcal{L}_i$. Then we have that:  
\begin{equation}\label{eqL}
 \sum_{l \in L_i} d_H(l) = |L_i|(q - 1) = (q^2 + 2q + 1)(q - 1) = q^3 + q^2 - q - 1.
\end{equation}
Consider the bipartite projection $G_{B_i}$ (Definition \ref{def:projection}). The incidence graph of $\mathcal{Q}(4,q)$ has girth 8, so $G_{i_{B_i}}$ has no 3-cycles. The maximum number of edges in a triangle-free graph of order $|B_i|=2q+2$ is equal to $(q+1)^2$, and it is obtained by the complete bipartite graph $K_{q+1,q+1}$. Thus, $G_{i_{B_i}}\cong K_{q+1,q+1}$ and the distance between any two lines in $G_i$ is at most 4. 

Since $\mathcal{Q}(4,q)$ has girth 8, the degree in $H$ of any point $p' \in R$ is at most 1 (i.e., $d_H(p') \le 1$). This implies that the sum of the degrees of the points in $R$ is bounded by the total number of points in $R$:
\begin{equation}\label{eqP}
    \sum_{p' \in R} d_H(p') \le |R|  = q^3 - q^2 + q - 1. 
\end{equation}
Since $H$ is a bipartite graph, the sum of the degrees of the vertices in both partitions must be equal:
$\sum_{p' \in R} d_H(p') = \sum_{\ell \in \mathcal{L}_i} d_H(\ell)$. But then Equations \ref{eqL} and \ref{eqP} lead to a contradiction.

Analogously, if $|B_i|=2q+3$, since $G_i$ is bipartite,  $$\sum_{\ell\in\mathcal{L}_i} d_{G_i}(\ell)= \sum_{p\in B_i} d_{G_i}(p) \mbox{, and } |\mathcal{L}_i|=\sum_{p\in B_i} d_{G_i}(p)/2.$$ 
It follows that $|\mathcal{L}_i|=(2q+3)(q+1)/2$.
As previously, since $\mathcal{Q}(4,q)$ has girth 8, the degree in $H$ of any point $p' \in R$ is at most 1 (i.e., $d_H(p') \le 1$),
$$ \sum_{p' \in R} d_H(p') \le |R|  = q^3 - q^2 + q - 1,\text{ and }
\sum_{l \in L_i} d_H(l) = (q - 1)(2q+3)(q+1)/2.$$
This contradicts $ \sum_{p' \in R} d_H(p') = \sum_{\ell \in \mathcal{L}_i} d_H(\ell) $.
Thus, if $q\ge3$, every $B_i$ has order at least $2q+4$.
\end{proof}

In the following we will analyze the cases where $E^*\neq \emptyset$, that is, when {$|B|>2(q^2+1)$. Let $G^*_i=(B_i,\mathcal{L}_i)-E^*$ be a connected component of $G^*=(B,\mathcal{L})-E^*$.}
We define $E^*_i=E(G_i)\cap E^*.$

\begin{lemma}\label{lem:B-E*}
    Let $B$ be a $2$-blocking set of $Q(4,q)$ and consider a set $E^*$. 
    Let $G^*_i=(B_i,{\mathcal{L}}_i)$ be a component of $G^*=(B,\mathcal{L})-E^*$.
        \begin{enumerate}
        \item If $|B_i|=2q-2\alpha$ and $\alpha\ge0$, then $|E^*_i|\ge 2q$.
        \item If $|B_i|=2q-2\alpha-1$ and $\alpha\ge0$, then $|E^*_i|\ge 3q-1$.
       \item If $|B_i|=2q+1$, then $|E^*_i|= q+1$.
     \end{enumerate}
\end{lemma}

\begin{proof}
    For each case, let $G^*_i=(B_i,\mathcal{L}_i)$ be a component of $G^*=(B,\mathcal{L})-E^*$. Let $G_{i_{B_i}}^*$ be the $B_i$-projection of $G_i^*$. Since each line of $G_i^*$ has degree 2, there exists a bijection between the lines of $G_i^*$ and the edges of $G_{i_{B_i}}^*$. Since $\mathcal{Q}(4,q)$ has girth 8, it follows that $G_{i_{B_i}}^*$ has no triangles. Recall that the maximum number of edges in a triangle-free graph is attained by the (quasi-)balanced bipartite graph. 

    \begin{itemize}
    
    \item {\bf{Case 1.}} Let $|B_i|=2q-2\alpha>0$ with $0\leq \alpha<q$. 
    
    Since $|V(H)|=|B_i|$, it follows that $|E(H)|\le(q-\alpha)^2$, $\sum_{p\in B_i} d(p)=(q+1)(2q-2\alpha)$ and $\sum_{\ell\in \mathcal{L}_i} d(\ell)=2(q-\alpha)^2$. Thus,
    $$|E^*|\ge(q+1)(2q-2\alpha)-2(q-\alpha)^2=2q(q-\alpha)\ge2q.$$ 
    
    \item {\bf{Case 2.}} Let $|B_i|=2q-2\alpha-1\ge1$ and $0\le\alpha\le q-1$. 
    
    Since $|V(H)|=|B_i|$, it follows that $|E(H)|\le(q-\alpha)(q-\alpha-1)$, $\sum_{p\in B_i} d_{G_i}(p)=(q+1)(2q-2\alpha-1)$. Thus, 
    $$|E^*|\ge(q+1)(2q-2\alpha-1)-2(q-\alpha)(q-\alpha-1)=3q+2\alpha(q-\alpha-2)-1\ge 3q-1.$$
    
    \item {\bf{Case 3.}} If $|B_i|=2q+1$.
    
    Since $|V(H)|=|B_i|$, it follows that $|E(H)|\le q(q+1)$ and $\sum_{p\in B_i} d(p)=(q+1)(2q+1)$. Thus,
     $|E^*|\ge(q+1)(2q+1)-2q(q+1)=q+1.$
     \end{itemize}
\end{proof}
Now, we are ready to {prove} Theorem \ref{upperboundthm}.

{\bf{ Proof of Theorem \ref{upperboundthm}: }}

\begin{proof}

{Consider a} rainbow-free coloring that attains the upper chromatic {number}. {The set  of the non singular chromatic classes of $B$} is  a  $2$-blocking set. {The} induced subgraph $G = (B, \mathcal{L})$  of the incidence graph of ${\mathcal{Q}(4,q)}$ {has}  the property that, if $B_i$ is a chromatic class of the coloring, then the subgraph $(B_i, N_G(B_i))$ is connected; otherwise, we could obtain a rainbow-free coloring of $\mathcal{Q}(4,q)$ with more colors. 
{As we described in the first part of this section, we can choose a set of edges $E^*$ such that $G^*=G-E^*$
has the property that each connected component $G^*_i=(B_i, \mathcal{L}_i)$ of $G_i^*$ covers its lines (each line has exactly two neighbors in $B_i$)}. Recall that the set of edges $E^*$ may not be unique, and that we defined $E^*_i=E(G_i)\cap E^*.$

By Lemma \ref {blockingset}, any $2$-blocking set satisfies that $|B| \ge 2(q^2 + 1)$. Since $|\mathcal{P}| = q^3 + q^2 + q + 1$ and $|B| \ge 2q^2 + 2$, the set $|R| \leq q^3 - q^2 + q - 1,$ where $R$ is the set of singleton chromatic  classes.

In $G^*$, every component is nontrivial and contains only points of the same color; consequently, if $m$ is the number of connected components of $G^*$, we have that the total number of colors of $\mathcal{P}$ is $m + |R|$.
{For the sake of contradiction, assume that $\overline{\chi}(\mathcal{Q}(4,q))\geq q^3 -q^2 +2q-2 $. By the definition of $\overline{\chi}$, 
if every coloring with $c$ colors has a rainbow line, then every $c'$-coloring, with $c'\geq c$, also has a rainbow line, because we can identify chromatic classes in the $c'$-coloring, obtaining a $c$-coloring and a rainbow line in a $c$-coloring. Clearly this line is a rainbow line in the original $c'$-coloring}. Thus, it suffices to prove that in a coloring with $q^3 -q^2 +2q-2 $ colors, the union of the non-singular chromatic classes does not contain a 2-blocking set. Assume that it does contain a 2-blocking set $B$. By Lemma \ref {blockingset}, any $2$-blocking set satisfies that $|B| \ge 2(q^2 + 1)$. Since $|\mathcal{P}| = q^3 + q^2 + q + 1$ and $|B| \ge 2q^2 + 2$, the set $|R| \leq q^3 - q^2 + q - 1,$ where $R$ is the set of singleton chromatic classes.
Thus,  {$|R| + m =c = q^3 - q^2 +2q -2 $.} Since $|R| \le q^3 - q^2 + q - 1$, it follows that $m \ge q - 1$, that is $G^*$ has at least $q-1$ components. 

We analyze the possible values for $m$, recall that $|R| = q^3 - q^2 + 2q - 2 - m$ and $|B| = q^3+q^2 +q+ 1-|R|$.

\begin{itemize}
\item {\bf{Case $\mathbf{m = q - 1}$:}}
In this case, $|R| = q^3 - q^2 + q - 1$,  $|B| = 2q^2 + 2$ and $E^* =\emptyset$.

Since $|B|=2q^2+2$ and $G{^*}$ has {exactly}  $q-1$ components, {by Lemma \ref{Lem:CotaB_q-1},} every $B_i$ has order at least $2q+4$.
Suppose that every component $G_i$ satisfies $|B_i|\geq 2q+4$.
{Since there are $q-1$ components, this would imply $2q^2+2=\sum_i |B_i| \geq (2q+4)(q-1).$
However, for $q\geq 4$ we have $(2q+4)(q-1)=2q^2+2q-4>2q^2+2$. Hence, some component $G_i$ satisfies $|B_i|<2q+4$, 
a contradiction.}  

\item {\bf{Case $\mathbf{m = q}$:}}
It follows that $|R| = q^3 - q^2 + q - 2$,   $|B| = 2q^2 + 3$, and $|E^*| = |B|(q+1) - 2|\mathcal{L}| = (2q^2+3)(q+1) - 2(q^3+q^2+q+1) = q + 1$.
The average order of a component is $\frac{2q^2+3}{q} = 2q + \frac{3}{q}$, and thus, there exists at least one component $G_\alpha$ of order at most $2q$. By Lemma \ref{lem:B-E*}, such a component requires that $|E^*_\alpha|\ge2q$. Therefore, $|E^*| \ge 2q$, contradicting  $|E^*| = q + 1$.

\item {\bf{Case $\mathbf{m \ge q + 1}$:}}.
We have $|R| {=} q^3 - q^2 + 2q - 2 - m$, $|B| {=} 2q^2 - q + 3 + m$, and $|E^*| {=} (2q^2 - q + 3 + m)(q+1) - 2(q^3+q^2+q+1) = (q+1)(m - q + 1)$.

We consider the order of the components in $G^*$:

If $G^*$ has at least $m - (q-1)$ components of order at most $2q$, then by Lemma \ref{lem:B-E*}, $|E^*| \ge 2q(m - q + 1) > (q+1)(m - q + 1)$ for $q \ge 3$, which is a contradiction.

If $G^*$ has at least $m - (q-1)$ components of order at most $2q+1$, by {Lemma \ref{lem:B-E*}}, it must have exactly $m - (q-1)$ components of order $2q+1$, and the remaining $m - (m - (q-1)) = q - 1$ components must be of order at least $2q+2$. 
In this case, $|B| \ge (q-1)(2q+2) + (m - q + 1)(2q+1) = 2mq + m + q - 1$. 
Since $m \ge q + 1$, we have $|B| \ge 2q^2 + 5q + m - 1 > 2q^2 - q + 3 + m$, a contradiction.
 
Thus, $G^*$ must have at least $q$ components of order at least $2q+2$. In this case, 
    $|B| \ge q(2q+2) + (m - q) = 2q^2 + m + q > 2q^2 - q + 3 + m$ for $q \ge {2}$, which is a contradiction.
    
\end{itemize}
Since all possible cases for the number of components lead to a contradiction, we conclude that $\overline{\chi}(\mathcal{Q}(4,q)) \leq q^3 - q^2  +2q -3.$
\end{proof}

{Combining the lower bound established in Theorem \ref{lowerboundsthm} with the upper bound given in Theorem \ref{upperboundthm}, we obtain the following result.}

\begin{theorem}\label{cor:cotas}
If $q$ is an even prime power or $q=3$, then
$$q^3-q^2+q\le\overline{\chi}(\mathcal{Q}(4,q))\le  q^3-q^2+2q-3.$$
\end{theorem}

\section{Generalized  Quadrangles of order $2$ and $3$}
\label{q2y3}

In this section, we determine the upper chromatic number of the classical generalized quadrangle $\mathcal{Q}(4,q)$, for $q=2$ and {for} $q=3$. 

We {start} with {the case} $q=3$, because we use similar techniques as the ones used in the proof of Theorem \ref{upperboundthm}.

Finally, we prove the Theorem {\ref{mainthm}} for $q=2$. Although the techniques and arguments used to prove Theorem \ref{upperboundthm} could readily be applied again, we provide an alternative proof that we find both instructive and elegant. This proof makes use of Sylvester’s description of the generalized quadrangle of order $2$, which is based on a decomposition of the edges of $K_6$ into $1$-factors. We believe that presenting this decomposition and the associated proof is both natural and worthwhile, particularly since it constituted the starting point of the present work several years ago.

\begin{theorem}\label{teo:q=3}
    $ \overline{\chi}(\mathcal{Q}(4,3))=21$.
\end{theorem}

\begin{proof}
    By Lemma \ref{lowerboundsthm}, $ \overline{\chi}(\mathcal{Q}(4,3))\ge21$.
    For the upper bound, suppose that there exists a $22$-coloring rainbow-free coloring of $\mathcal{Q}(4,q)$.
    Notice that, if we take all the chromatic classes with more than one element, we have a set $B\subset \mathcal{P}$ that is a $2$-blocking set of ${\mathcal{Q}(4,q)}$.
    By Lemma \ref{blockingset}, $|B|\ge20$. 
    If $G=(B,\mathcal{L})$ is connected, then $G$ uses one color, and since there are at most 20 singular classes in $\mathcal{P}$, the coloring uses 21 colors, a contradiction.
    
    Thus, $G$ is not connected. Since each component $G_i$ uses a different color, if $G$ has $m\ge2$ components, then $20\le|B|=m+18$. Let $R=\mathcal{P}\setminus B$.

    If $m=2$, then $|B|=20$. Let $B=B_1\cup B_2$, by Lemma \ref{Lem:CotaB_q-1}, $|B_i|\ge10$.
    Thus $|B_1|=|B_2|=10$ and by Lemma \ref{Lem:CotaB} and a double counting argument $|\mathcal{L}_1|=|\mathcal{L}_2|=10$. In what follows, we analyze the maximum degree of a vertex in $R$ with respect to the set of lines $\mathcal{L}_i$. Suppose there is a vertex $p\in R$ such that $d_{\mathcal{L}_i}(p)\ge3$, and let $\ell_1,\ell_2,\ell_3\in N_{\mathcal{L}_i}(p)$, since the girth of $\mathcal{Q}(4,3)$ is 8, it follows that the $dist(\ell_j,\ell_k)\ge6$   for $j,k\in\{1,2,3\}$, and thus $N_{B_i}(\ell_j)\cap N_{B_i}(\ell_k)\setminus\{p\}=\emptyset$, else $(p,\ell_j,p',\ell_k,p)$ is a 4-cycle, with $p'\in N_{B_i}(\ell_j)\cap N_{B_i}(\ell_k)\setminus\{p\}$. 
    Let $N_{B_i}(\ell_1)=\{p_1,p_2\}$, $N_{B_i}(\ell_2)=\{p_3,p_4\}$, $N_{B_i}(\ell_3)=\{p_5,p_6\}$. Since each point in $B_i$ has degree 4 and the girth of $\mathcal{Q}(4,3)$ is 8, then $N_{\mathcal{L}_i}(p_j)\cap N_{\mathcal{L}_i}(p_k)\setminus\{\ell_1,\ell_2,\ell_3\}=\emptyset$, else $(p,\ell_j,p_{j'},\ell',p_{k'},\ell_k,p)$ is a 6-cycle, with $p_{j'}\in N_{B_i}(\ell_j)\setminus\{p\}$, $p_{k'}\in N_{B_i}(\ell_k)\setminus\{p\}$, and $\ell'\in N_{\mathcal{L}_i}(p_{j'})\cap N_{\mathcal{L}_i}(p_{k'})\setminus\{\ell_1,\ell_2,\ell_3\}$. 
    Therefore, each point $\{p_1,p_2,\dots,p_6\}$ has at least another 3 neighbors, and thus, there are at least another $3*6=18$ lines in $\mathcal{L}_i$ different from $\ell_1,\ell_2,\ell_3$, contradicting that $|\mathcal{L}_i|=20$. Thus, for each point $p\in R$, $d_{\mathcal{L}_i}(p)=2$.

    Each vertex $p\in R$ has exactly two neighbors in each $\mathcal{L}_i$, and since $\mathcal{Q}(4,3)$ has girth 8, these two vertices are at distance at least 6, thus in each $G^*_i$ there are at least 20 pairs of vertices of ${\mathcal{L}_i}$ at distance at least 6 (one pair for each vertex $p\in R$). 
    The vertex-set ${\mathcal{L}_1}$ has order 20 and there are 20 pairs of vertices at distance at least 6, so, there is at least one vertex  $\ell\in\mathcal{L}_1$ with two vertices $\ell_1,\ell_2\in{\mathcal{L}_1}$ at distance 6, that is $dist(\ell,\ell_1),dist(\ell,\ell_2)\ge6$.  
    Consider the vertex $\ell$, by Lemma \ref{Lem:CotaB}, let $\{p_1,p_2\}=N_{G_1}(\ell)$, and by Lemma \ref{Lem:CotaB}, let $N_{G_1}(p_1)=\{\ell_3,\ell_4,\ell_5\}$ and let $N_{G_1}(p_2)=\{\ell_6,\ell_7,\ell_8\}$, since $\mathcal{Q}(4,3)$ has girth 8, $N_{G_1}(p_1)\cap N_{G_1}(p_2)=\emptyset$. For $3\le j\le8$, let $p_j=N_{G_1}(\ell_j)\setminus\{p_1,p_2\}$, and  $p_j\neq p_k$ for $j\neq k$. 
    Now, consider the vertices $\ell_1,\ell_2$, since $dist(\ell,\ell_1),dist(\ell,\ell_2)\ge6$, it follows that $N_{G_1}(\ell_1)\cap\{p_1,p_2,\dots,p_8\}=\emptyset$ and $N_{G_1}(\ell_2)\cap\{p_1,p_2,\dots,p_8\}=\emptyset$. Since $|N_{G_1}(\ell_1)\cup N_{G_1}(\ell_2)|\ge3$, $\{p_1,p_2,\dots,p_8\}\cup N_{G_1}(\ell_1)\cup N_{G_1}(\ell_2)\subset B_1$ and $|B_1|\ge11$, a contradiction.
    
    If $m=3$, then $|B|=21$ and $E^*\neq\emptyset$. Let $G^*=(B,\mathcal{L})-E^*$, since $\sum_{p\in B}d_{G^*}(p)=4(21)$ and $\sum_{\ell\in \mathcal{L}}d_{G^*}(\ell)=2(40)$, it follwos that $|E^*|=4=q+1$. In this case, by Lemma \ref{lem:B-E*}, we have no components of order less than $2q+1=7$, and since $|E^*|=4$, we have at most one component of order 7. 
    Thus $|B|\ge2(8)+7=23$, a contradiction.
    
    For $m\ge4$ connected components, the blocking-set has order $|B| = 18+ m $ and there are $|E^*|=4(m-2)$ edges in $G=(B,\mathcal{L})$ not needed to cover $\mathcal{L}$. By Lemma \ref{lem:B-E*}, we can have at most  $m-2$ connected components of order less than or equal to 7, and each of these components has exactly 4 edges in $E^*$. Thus, each of these components has order 7 and $|B|\ge 7(m-2)+16=7m+2$. In this case, we use at most $38-6m\le 14$ colors, a contradiction.
\end{proof}

As we {mentioned at the beginning} of this section, although the techniques and arguments used to prove Theorem \ref{upperboundthm} could readily be applied again, we will share an alternative proof for $q=2$ that we consider different and elegant. 
To give this proof, we will describe $\mathcal{Q}(4,2)$ with a nice representation given by Sylvester in \cite{sylvester}. 

Let $G$ be the complete graph $K_6$. We will take the set of points of $\mathcal{Q}(4,2)$ as the edges of $G$, and the lines of $\mathcal{Q}(4,2)$ as the $1$-factors of $K_6$. 
With this model, we have exactly $15$ points and $15$ lines and it is not dificult to prove that for any point $p\in \mathcal{P}(\mathcal{Q}(4,2))$ and any line $l\in \mathcal{L}(\mathcal{Q}(4,2))$, that not contians $p$, there exist one and only one line $l'\in \mathcal{L}(\mathcal{Q}(4,2))$ such that $l\cap l'\not = \emptyset$. To see that, notice that if we have an edge $ij\in E(K_6)$ and any $1$-factor $F$ that not contains $ij$, then $F=\{ik,jm,st\}$, and immediately the $1$-factor $F'=\{ij, km, st\}$, satisfyes that $F \cap F'=\{st\} $. 

Now, we can prove the following theorem.

\begin{theorem}\label{thm:Q42}
$\overline{\chi}(\mathcal{Q}(4,2)) = 6$.
\end{theorem}

\begin{proof}
Using this description of $\mathcal{Q}(4,2)$, it is not difficult to see that $\mathcal{O}=\{\, e\in E(K_6)\colon v_1\in e \,\}.$ Since no perfect matching in $K_6$ contains two edges of $\mathcal{O}$, the set $\mathcal{O}$ meets every line of $\mathcal{O}$ in exactly one point and has exactly five elements. Hence $\mathcal{O}$ is an ovoid of $\mathcal{Q}(4,2)$. 
Moreover, the rest of the edges of $K_6$, that we call ${\mathcal{O}}_2$, is a $ 2$-ovoid, that is the complement of $\mathcal{O}$. 
$$\mathcal{O}_2=\{\, e\in E(K_6)\colon v_1\notin e \ \}.$$ Observe that ${\mathcal{O}}_2$ are the ten edges of $K_5=K_6-\{v_1\}$.
Notice that any factor of $K_6$ intersects this $K_5$ into two edges.

With this description, or equivalently by applying Remark~\ref{2blocking}, we obtain a $6$-coloring of $\mathcal{Q}(4,2)$ by assigning one color to all edges of $\mathcal{O}_2$, and assigning distinct colors to each edges of $\mathcal{O}$, using colors different from the one assigned to $\mathcal{O}_2$.

For the upper bound, by contradiction, assume that there exists a rainbow–free $7$–coloring of $\mathcal{Q}(4,2)$. Then $K_6$ contains a set $S=\{e_1,\dots,e_7\}$ of seven pairwise differently colored edges, and let $H$ be the subgraph induced by $S$. Since $H$ has seven edges on six vertices, it cannot be acyclic, and we show that any such subgraph forces a heterochromatic matching of size three.

If $H$ contains a $6$–cycle, three pairwise disjoint edges of the cycle already form a heterochromatic matching. If $H$ contains a $5$–cycle, then the seventh edge either connects the cycle to the remaining vertex of $K_6$, producing three disjoint colored edges, or it connects two vertices of the cycle, creating a $4$–cycle; in both situations we can select three pairwise disjoint edges of distinct colors, and therefore obtain a heterochromatic matching. 

Finally, if $H$ contains no cycle of length at least four, then $H$ consists of two triangles. If they share a vertex, the seventh edge necessarily creates a longer cycle, reverting to a previous case; if they are vertex–disjoint, any edge between them together with one edge from each triangle yields a heterochromatic matching.

In all possibilities, we obtain a heterochromatic matching of size three, contradicting the assumption that the coloring is rainbow–free in $\mathcal{Q}(4,2)$. Hence no such $7$–coloring exists, so $
\overline{\chi}(\mathcal{Q}(4,2)) \le 6$.

Then: $\overline{\chi}(\mathcal{Q}(4,2))= 6$.
\end{proof}

\section{Conclusions, open problems and future work} \label{sec.conclandfuture}

As we said in the introduction of this paper, we go a step further related to the question about the upper chromatic number of generalized polygons. It is well known that the generalized triangles \cite{KS20} are the projective planes, and the upper chromatic number for projective planes has been studied by different authors (see \cite{daisyPP,bhs, bt}).

In our opinion, we left interesting open problems: 
\begin{itemize}
\item{\bf{Problem 1:}} Improve the lower and upper bounds of issue (1) in Theorem \ref{mainthm}.
\item{\bf{Problem 2:}} Improve the upper bound of issue (2) in Theorem \ref{mainthm}.
\item{\bf{Problem 3:}} Find the exact value for $q=\{4,5\}$ in Theorem \ref{mainthm}.
\end{itemize}

An interesting direction for future research is to determine lower and upper bounds, as well as exact values, for generalized hexagons of prime power order $q$.


\end{document}